\documentclass[a4paper,11pt,oneside]{article}
\usepackage[ansinew]{inputenc}
\usepackage{amsfonts} 
\usepackage{amsthm} 
\usepackage{amsmath} 
\usepackage{amssymb}
\usepackage{verbatim}
\usepackage{graphicx}
\usepackage[all,cmtip]{xy}

\reversemarginpar
\newcommand{\R}{{\mathbb R}}
\newcommand{\N}{{\mathbb N}}
\newcommand{\Z}{{\mathbb Z}}

\newcommand{\C}{{\mathbb C}}
\newcommand{\T}{{\mathbb T}}
\newcommand{\ST}{ S_{_{TG}}(\T)}

\newtheorem{theorem}{Theorem}[section]
\newtheorem{lemma}[theorem]{Lemma}
\newtheorem{proposition}[theorem]{Proposition}
\newtheorem{corollary}[theorem]{Corollary}
\theoremstyle{definition}\newtheorem{remark}[theorem]{Remark}
\theoremstyle{definition}\newtheorem{example}[theorem]{Example}
\theoremstyle{definition}\newtheorem{definition}[theorem]{Definition}
\theoremstyle{definition}

\title{Extensions of topological Abelian groups and three-space problems}
\date{}
\author{
Hugo J. Bello \\ {\small\em Departamento de F\'{\i}sica y Matem\'{a}tica Aplicada,}\\ {\small\em Universidad  de Navarra, Spain.}\\
{\small\em e-mail: hbello.1@alumni.unav.es}\\ \\
 Mar\'{\i}a Jes\'us Chasco \\ {\small\em Departamento de F\'{\i}sica y Matem\'{a}tica Aplicada,}\\ {\small\em Universidad  de Navarra, Spain.}\\ 
{\small\em e-mail: mjchasco@unav.es}\\ \\
 Xabier Dom\'{\i}nguez\\ {\small\em Departamento de M\'{e}todos Matem\'{a}ticos y de
Representaci\'{o}n,} \\{\small\em Universidad de A Coru\~{n}a, Spain.} \\{\small\em e-mail: xabier.dominguez@udc.es}
}

\begin{document}
\maketitle
\begin{abstract}
 A twisted sum in the category of topological abelian groups is
a short exact sequence  $0\to Y\to X \to Z\to 0$ where all maps are assumed to be continuous and open onto their images. The twisted sum splits if it is equivalent to $0\to Y\to Y \times Z \to Z\to 0$.
\par
We study the   class $\ST$ of topological groups $G$ for which  every  twisted sum $0\to \T\to X \to G\to 0$ splits. We prove that this class contains locally precompact groups, sequential  direct limits of locally compact groups and topological groups with $\mathcal{L}_\infty$  topologies. We also prove that it is closed by taking open and dense subgroups, quotients by dually embedded subgroups and coproducts. As a technique to find further examples of groups in $\ST$ we use the relation of this class  with the existence of  quasi-characters on $G$ and with  three-space problems for topological groups.
The subject is inspired on some concepts  known  in the framework of topological vector  spaces such as  the notion of $K$-space, which were interpreted for topological groups by Cabello.
\end{abstract}

\section{Introduction and Preliminaries}

In the theory of topological vector spaces (topological groups) a property $P$ is said to be a {\it 3-space property} if whenever a closed subspace (subgroup) $Y$ of a space (group) $X$ and the corresponding quotient $X/Y$ both have property $P,$ $X$  also has property $P$.

A short exact sequence of spaces (groups) $0\to Y\stackrel{\imath}{\to} X \stackrel{\pi}{\to}Z\to 0$ will be called {\it twisted sum} and the space (group) $X$ will be called an {\em extension }  of $Z$ by  $Y$. Both $\imath$ and $\pi$ are assumed to be continuous and open onto their images. Using this language, 3-space properties can be described as those which are preserved by forming extensions.

An example of a 3-space property in the category of Banach spaces is reflexivity. However, the point separating property (i.e. having a dual space which separates points) is not a 3-space property. (Consider the space $l_p$ for $0<p<1$, and $M$ a weakly closed subspace of $l_p$ without the Hahn-Banach extension property. If we take as $N$ the kernel of some continuous linear functional on $M$ which does not extend to $l_p$, then $X=l_p/N$ does not have the point separating property but both $L:=M/N$ and $X/L$ have this property (see \cite{kaltonsampler}).)

In the category of topological Abelian groups, local compactness, precompactness, metrizability and completeness are 3-space properties. However, $\sigma$-compactness, sequential completeness, realcompactness and a number of other properties are not (see \cite{brutka} for more examples).

The  twisted sum $0\to Y\stackrel{\imath}{\to} X \stackrel{\pi}{\to}Z\to 0$ {\em splits} if there exists a continuous linear map (continuous homomorphism) $T:X\to Y\times Z$ making the following diagram commutative ($\imath_Y$ is the canonical inclusion of $Y$ into the product, and $\pi_Z$ is the canonical projection).

\[
\xymatrix{
             &                                    & X\ar[dd]^T\ar[dr]^{\pi} &             & \\
 0\ar[r] & Y\ar[ur]^\imath\ar[dr]_{\imath_Y} &                                         & Z\ar[r] &0\\
            &                                    & Y\times Z\ar[ur]_{\pi_Z}                 &             &
}
\]

It is known that if such a $T$ exists, it must actually be a topological isomorphism.

 Of course, if the twisted sum $0\to Y\stackrel{\imath}{\to} X \stackrel{\pi}{\to}Z\to 0$  splits and  both $X$, $Z$ have a productive property $P,$ then $X$ has property $P$, too.

 Kalton, Peck and Roberts provided in \cite{kaltonsampler} the first formal and extensive study of splitting  twisted sums in the framework of {\em $F$-spaces} (complete metric linear spaces).
 They devote Chapter $5$ of this monograph to the following  problem: Is local convexity a 3-space property?

On the path to answering (in the negative) this question, the authors mention a useful result by Dierolf (1973) which asserts that there exists a non-locally-convex extension of  $Z$ by $Y$ if and only if there exists a non-locally-convex extension of $Z$ by $\R$. The analogue of this result for topological groups, which involves the notion of local quasi-convexity, was obtained by Castillo in \cite{castillo}.

At this point the following definition, originally introduced in \cite{kaltonpeck}, comes across as natural: An $F$-space $X$  is said to be a {\em $\mathcal K$--space}, if whenever $Y$ is an $ F$-space and $L$ is a subspace of $Y$ with dimension one such that $Y/L \cong X$, the corresponding twisted sum splits. The negative answer to the 3-space problem for local convexity is obtained in \cite{kaltonsampler} as a corollary of the fact that $\ell_1$ is not a $\mathcal K$--space.

The notion of $\mathcal K$--space is relevant on its own, regardless of 3-space properties. Many classical spaces such as $L_p$ ($0<p<\infty$), $l_p$ ($p\neq 1$), or $c_0$ are $\mathcal K$--spaces.

In this paper we will study the natural counterpart of the notion of $\mathcal K$--space for topological groups, and its connections with the 3-space problem, following the work started by Cabello in \cite{cabello2000},  \cite{cabellojmaa}, \cite{cabello}.

For simplicity, and because our methods are applicable for the most part only to
Abelian groups, we use additive notation, and denote by $0$ the neutral element.

We denote by $\N$ the set of all positive natural numbers, by $\Z$ the integers, by $\R$ the reals, by $\C$ the set of complex numbers and by $\T$  the unit circle of $\C$, with the topology induced by $\C$. In $\T$ we will use multiplicative notation and  we will denote by $p$ the canonical projection from $\R$ to $\T$ given by $p(t)=e^{2\pi it}$. We will  use  $\mathcal{N}_0(G)$ to denote the system of neighborhoods of $0$ in a  topological Abelian group $G$.
Recall that an Abelian topological group $G$ is precompact if for every neighborhood of zero $V$
there exists a finite subset $F$ of $G$ such that $G = F + V$. Precompact groups are the subgroups of compact groups. In the same way a group $G$ is locally precompact if and only if it is a subgroup of a locally compact group.

Recall that the {\em dual group $G^{\wedge}$} of a given Abelian topological group $G$ is formed from the continuous group homomorphisms from $G$ into $\mathbb{T}$, usually called {\em characters}.
The polar set of a subset $A$ of $G$ is defined by $A^0 := \{\chi \in G^\wedge : \chi (A)\subseteq \T_+\}$, where $\mathbb{T}_+ :=p([-1/4,1/4])=\{z\in \mathbb{T}:Re(z)\geq 0\}$. A subset $A$ of a topological group $G$ is called \emph{quasi-convex} if for every $x\in G\setminus A$ there is a $\chi \in A^0$ such that $\chi (x)\notin \T_+$.
A topological group is called \emph{locally quasi-convex} if it has a neighborhood basis of $0$ consisting of quasi-convex sets. It is well known (see \cite{Ban}) that a topological vector space is locally convex if and only if it is a locally quasi-convex topological group in its additive structure.

If $G^\wedge$ separates the points of $G$ we say that $G$ is {\em maximally almost periodic (MAP)}. Every locally quasi--convex group is a MAP group.

The analogous notion to the HBEP (Hahn-Banach Extension Property) for topological groups is the following:  A subgroup $H\leq X$ is {\em dually embedded} in $X$ if each character of $H$ can be extended to a character of $X$.

When endowed with the compact-open topology $ \tau_{co}$, $G^{\wedge}$ becomes a Hausdorff topological group. A basis of
neighborhoods of the neutral element for the  compact open topology $\tau_{co}$ is given by the sets $K^0 = \{\chi \in
G^{\wedge}: \chi(K) \subseteq \mathbb{T}_+\}$, where $K$ is a compact subset of $G$.

\begin{remark} Observe that a necessary condition for the splitting of
  the twisted sum of topological Abelian groups $0\to H\stackrel{\imath}{\to} X \stackrel{\pi}{\to}G\to 0$, is that  $i(H)$ be a dually embedded subgroup of $X$.
\end{remark}
The following known characterization is essential when dealing with twisted sums in different categories:
\begin{theorem}\label{theorem_section} Let $0\to H\stackrel{\imath}{\to} X \stackrel{\pi}{\to}G\to 0$ be a twisted sum of topological Abelian groups. The following are equivalent:
\begin{enumerate}
\item $0\to H\stackrel{\imath}{\to} X \stackrel{\pi}{\to}G\to 0$ splits.

\item There exists a continuous homomorphism $S:G\to X$ with $\pi\circ S={\rm id}_G,$ i.e. a right inverse for $\pi.$ 
\item There exists a continuous homomorphism $P:X\to H$ with $P\circ \imath={\rm id}_H,$ i.e. a left inverse for $\imath.$  

\end{enumerate}
\end{theorem}

We will use the notions of pull-back and push-out in the category of topological Abelian groups, following Castillo \cite{castillo}:  Given topological Abelian groups $A,$ $B$ and $C$ and continuous homomorphisms $u:A\to B$ and $v:A\to C$, the push-out of $(u,v)$ is a topological group $PO$ and two continuous homomorphisms $r$ and $s$ making the diagram commutative
\[
\xymatrix{
A\ar[r]^{u}\ar[d]_v    &B \ar[d]^s \ar@/^/[rdd]^{s'}& \\
C\ar[r]_r \ar@/_/[rrd]_{r'}    &PO\ar@{-->}[rd]^{\phi} & \\
& & PO'
}
\]
and such that for every topological Abelian group $PO'$ and continuous homomorphisms $r':C\longrightarrow PO'$ and $s':B\longrightarrow PO'$  with $s'\circ u = r'\circ v$,  there is a unique continuous homomorphism $\phi$ from $PO$ to $PO'$ making the two triangles commutative.
The topological group $PO$ exits and is unique up to topological isomorphism.

Given any twisted sum $0\to H\stackrel{\imath}{\to} X \stackrel{\pi}{\to}G\to 0$ of topological Abelian groups, any topological Abelian group $Y$ and any continuous homomorphism $t:H\to Y,$ if $PO$ is the push-out of $\imath$ and $t,$ there is a commutative diagram

\begin{equation*}
\xymatrix{
0\ar[r]   &H\ar[r]^\imath\ar[d]^t       &X\ar[r]^\pi\ar[d]^s          &G\ar@{=}[d]\ar[r]  &0\\
0\ar[r]   &Y\ar[r]_r                    &PO\ar@{-->}[r]_p                         &G\ar[r]  &0
}
\end{equation*}
where
both squares are commutative and the bottom sequence is a twisted sum \cite{castillo}.

An analogous result for the dual construction (the pull-back) can be obtained (see \cite{castillo}).

\begin{lemma}\label{bowie} Let $K$ be a compact subgroup of a topological Abelian group $X$.  $X^{\wedge}$ separates points of $K$ if and only if $K$ is dually embedded in $X$.\end{lemma}
\begin{proof}Suppose that $X^\wedge$ separates points of $K$. It is known that for any locally compact Abelian group $G$, a subgroup $L\le G^{\wedge}$ is dense in $G^{\wedge}$ if and only if it separates points of $G$ (see Proposition 31 in \cite{morris}). The subgroup $L$ of $K^{\wedge}$ formed by all restrictions of characters of $X$ separates points of $K$ by hypothesis. Hence $L$ is dense in $K^{\wedge}$ and, as $K^{\wedge}$ is discrete, $L$ coincides with $K^{\wedge}.$

Suppose that $K$ is dually embedded in $X$. As $K$ is compact, it is a MAP group. Fix a nonzero $x\in K$. There exists $\chi \in K^\wedge$ such that $\chi(x)\not= 1$. Since $K$ is dually embedded in $X$, there exists an extension $\widetilde{\chi}\in G^\wedge$ of $\chi$ with $\widetilde{\chi} (x)=\chi(x)\not=1$.
\end{proof}
\begin{corollary}(\cite[Proposition 1.4]{BruPei}) Let $K$ be a compact subgroup of a topological Abelian group $X$. If $X$ has sufficiently many characters then $K$ is dually embedded in $X$.
\end{corollary}
\section{The class $\ST$}
Next we consider the particular case in which the compact subgroup is $\T$. (Theorems \ref{mathsmath} and \ref{lqcase} below are implicitly proved in Theorem 4.1 in \cite{castillo}, but we prefer the present formulation.)
\begin{theorem}\label{mathsmath} Let $0\to \mathbb T\stackrel{\imath}{\to} X \stackrel{\pi}{\to}G\to 0$ be a twisted sum of topological Abelian groups. The following are equivalent:
\begin{enumerate}
\item The twisted sum $0\to \mathbb T\stackrel{\imath}{\to} X \stackrel{\pi}{\to}G\to 0$ splits.
\item  $X^{\wedge}$ separates points of $\imath(\mathbb T)$.
\item $\imath(\mathbb T)$ is dually embedded in $X$.
\end{enumerate}
\end{theorem}
\begin{proof} $2\Leftrightarrow 3$ was proved in Lemma \ref{bowie}. $2 \Rightarrow 1$: Suppose that $X^{\wedge}$ separates points of $\imath(\mathbb T)$. By Lemma \ref{bowie}, $\imath(\mathbb T)$ is dually embedded in $X$, hence there exists a continuous character $\chi:X\to \mathbb T$ which extends the isomorphism $\varphi:\imath(\mathbb T)\to \mathbb T$ defined by $\varphi(\imath(t))=t.$ Since $\chi\circ \imath=\text{id}_{\mathbb T},$ the assertion follows from Theorem \ref{theorem_section}. $1 \Rightarrow 2$: Fix $x\in \imath(\T)$, $x=\imath(z)$ with $z\not=1$. By  Theorem \ref{theorem_section} there exists a continuous homomorphism $ P:X \to \T$  with $P\circ \imath = \text{id}_\T$, hence $P(\imath(z))=\text{id}_\T (z)=z\not= 1$.
\end{proof}
Now we will complete the previous Theorem under the assumption that $G$ is locally quasi-convex. We will use the following result due to Castillo, concerning the 3-space
problem in locally quasi-convex groups.
\begin{lemma} \label{casti} (\cite[Theorem 2.1]{castillo}) Let $H$ be a locally quasi-convex subgroup of a topological Abelian group $X$ such that $X/H$ is locally quasi-convex. Then $X$ is locally quasi-convex if and only if $H$ is dually embedded in $X$.
\end{lemma}

\begin{theorem}\label{lqcase} Let $0\to \mathbb T\stackrel{\imath}{\to} X \stackrel{\pi}{\to}G\to 0$ be a twisted sum of topological Abelian groups. Suppose that $G$ is locally quasi-convex. Then conditions  1, 2, 3 of the above theorem are equivalent to
\begin{itemize}
\item[4.] $X$ is locally quasi-convex.
\end{itemize}
\end{theorem}
\begin{proof}
 $1\Rightarrow 4$: If the twisted sum splits, $X$ is topologically isomorphic to the product of two locally quasi-convex grups, hence it is locally quasi convex. $4\Rightarrow 1:$ Given a twisted sum $0\to \mathbb T\stackrel{\imath}{\to} X \stackrel{\pi}{\to}G\to 0$ with $X$ and $G$ locally quasi-convex, since $G\cong X/\imath(\T)$, by Lemma \ref{casti} we deduce that $\imath(\T)$ is dually embedded in $X$, therefore by  Theorem \ref{mathsmath}  the twisted sum splits.
\end{proof}

Following the notation used by Doma\'nski \cite{domanski}, in the framework of topological vector spaces, we introduce the class $\ST$ which is the analogue of that of $\mathcal K$--spaces for topological Abelian groups.

\begin{definition}
We say that a topological Abelian group $G$  {\em belongs to  $\ST$} if  every twisted sum of topological Abelian groups  $0 \to \T \to X \to G \to 0$ splits.
\end{definition}

\begin{theorem}\label{l1magic}
Let $G$ be a topological vector space such that $G\in \ST$ as a topological group. Then $G$ is a $\mathcal K $-space.
\end{theorem}
\begin{proof}
Suppose that $G\in \ST.$ Fix a twisted sum of topological vector spaces $0\to \mathbb R\stackrel{\imath}{\to} X \stackrel{\pi}{\to}G\to 0$. Recall that we denote by
$p:\R \to \T$ the canonical projection, defined by $p(r):=\exp(2\pi ir)$.
If we consider the push-out $PO$ of $p$ and $\imath$, we obtain a commutative diagram
\[
\xymatrix{
0\ar[r]   &\R\ar[r]^\imath\ar[d]^p   &X\ar[r]^\pi\ar[d]^s        &G\ar@{=}[d]\ar[r]  &0\\
0\ar[r]   &\T\ar[r]_r           &PO\ar[r]\ar@{-->}@/^{5mm}/[l]^\varphi                &G\ar[r]  &0
}
\]
where the bottom sequence is exact.  Since $G\in \ST,$ this sequence splits. Hence there exists a left inverse $\varphi$ for $r$. Then $\varphi\circ s\circ\imath=p$. Since $X$ is a topological vector space, $\varphi \circ s$ is of the form $x\mapsto \exp(2\pi i f(x))$ for some $f\in X^*$, a continuous linear functional. This clearly implies $f\circ \imath={\rm id}_{\mathbb R}$, i. e. $f$ is a left inverse for $\imath,$ and hence the top sequence splits, too. \end{proof}

\begin{theorem}(\cite{kaltonpeck}, \cite{ribe}, \cite{roberts})\label{eleuno} There is a short exact sequence of topological vector spaces and continuous, relatively open linear maps $0\to \mathbb R\stackrel{\imath}{\to} X \stackrel{\pi}{\to} G \to 0$ which does not split. In other words, $\ell_1$ is not a $\mathcal K$-space.
\end{theorem}

Using Theorem \ref{l1magic} we deduce

\begin{corollary}\label{l1kindofmagic}
 $\ell_1\not\in \ST$.
\end{corollary}
\begin{remark}
 The above Corollary gives an example of a quotient $X/\mathbb T$ which is locally quasi--convex as a topological group but such that $X$ does not even separate points of $\mathbb T:$ a strong
failure of the 3-space property for   local quasi-convexity and for the property of being a MAP group.
\end{remark}

From Theorem \ref{mathsmath} it follows that a topological Abelian group $G$ is in $\ST$ if for every twisted sum of topological Abelian groups of the form $0\to \mathbb T\stackrel{\imath}{\to} X \stackrel{\pi}{\to}G\to 0$, the subgroup $\imath(\mathbb T)$ is dually embedded in $X$. Note that in any such twisted sum, for an arbitrary $A\le G$ the subgroup $\pi^{-1}(A)$ contains $\imath(\T).$ This yields the following criterion:

\begin{proposition}
\label{luzcasal2} Let $G$ be a topological Abelian group.
Suppose that there exists a subgroup $A\leq G$ which is in $\ST$ and satisfies the following property: For every twisted sum of topological Abelian groups of the form $0\to \mathbb T\stackrel{\imath}{\to} X \stackrel{\pi}{\to}G\to 0$, the subgroup $\pi^{-1}(A)$ is dually embedded in $X$. Then $G\in \ST$.
\end{proposition}

\begin{corollary} \label{opensbg}
Let $G$ be a topological Abelian group.
Suppose that there exists an open subgroup $A\leq G$ such that $A\in \ST$. Then $G\in \ST$.
\end{corollary}
\begin{proof} If $A$ is an open subgroup of $G$, with the notation of \ref{luzcasal2}, $\pi^{-1}(A)$ is an open subgroup of $X$, hence dually embedded.
\end{proof}
\begin{corollary}\label{desdedense} Let $G$ be a topological Abelian group.
Suppose that there exists a dense subgroup $A\leq G$ such that $A\in \ST$. Then $G\in \ST$.
\end{corollary}
\begin{proof} If $A$ is a dense subgroup of $G$, with the notation of \ref{luzcasal2}, $\pi^{-1}(A)$ is an dense subgroup of $X$ because $\pi$ is continuous and onto. In particular it is dually embedded in $X$.
\end{proof}

We next see that the converse of Corollary \ref{desdedense} is true in the metrizable case. For any topological Abelian group $G$, we denote by $\widetilde{G}$  the Raikov completion of $G$. See \cite{warner} for more information about this subject.

\begin{lemma}(\cite[Theorem 6.11]{warner})\label{nogishe}
If $G$ is a metrizable topological group that has a completion $\widetilde{G}$ and if $H$ is a closed normal subgroup of $G$, then $G/H$ has a completion that is topologically isomorphic to $\widetilde{G}/\widetilde{H}$, where $\widetilde{H}$ is the closure of $H$ in $\widetilde{G}$ .
\end{lemma}

\begin{proposition} Let $G$ be a  metrizable topological Abelian group which is in $\ST$. Suppose that $A$ is a dense subgroup of $G$. Then $A$ is in $\ST$, too.
\end{proposition}
\begin{proof} Let $X$ be a topological Abelian group and let $\T \cong L\le X$ be a subgroup such that $A\cong X/L.$ Since metrizability is a 3-space property (\cite[5.38(e)]{hr}), $X$ is a metrizable group. By Lemma \ref{nogishe}, $\widetilde{G}=\widetilde{A}\cong \widetilde{X}/L.$ By Corollary \ref{desdedense}, $\widetilde{G}\in \ST$. Thus $L$ is dually embedded in $\widetilde{X}$, hence it is dually embedded in $X,$ too.
\end{proof}

Our next aim is to prove that locally precompact groups are in $\ST.$ Note first that if $G$ is a topological Abelian group and  $H\le G$  is a precompact subgroup such that the quotient $G/H$ is locally precompact, then $G$ is locally precompact, too. Indeed, let $\pi:G\to G/H$ be the canonical projection. Choose $U\in \mathcal{N}_0(G)$ such that $\pi(U)$ is precompact. Let us see that $U$ is precompact. Given $V\in \mathcal{N}_0(G)$ we need to find a finite subset $F\subset G$ with $U\subset F+V.$ Fix $V'\in \mathcal{N}_0(G)$ with $V'+V'\subset V$ and find a finite $F_1\subset G$ with $H\subset F_1+V'.$ Since $\pi(U)$ is precompact there exists a finite subset $F_2'$ of $G/H$ with $\pi(U)\subset F_2'+\pi(V').$ We may suppose that $F_2'=\pi(F_2)$ being $F_2$ a finite subset of $G$. Hence $\pi(U)\subset \pi(F_2)+\pi(V')=\pi(F_2+V')$, which implies $U\subset\pi^{-1}\pi(U)\subset \pi^{-1}\pi(F_2+V')=F_2+V'+H\subset F_1+F_2+V'+V'\subset F+V$ if we put $F=F_1+F_2.$

\begin{theorem}
Locally precompact Hausdorff Abelian groups belong to  $\ST$.
\end{theorem}
\begin{proof} Let $G$ be a locally precompact Abelian group. Given a twisted sum  $0\to \mathbb T\stackrel{\imath}{\to} X \stackrel{\pi}{\to}G\to 0$, as $\T$ is in particular precompact, $X$ is locally precompact.  But every subgroup of a locally compact group is dually embedded  \cite{hr}, so $\imath(\T)$ is dually embedded in $\widetilde X$, the completion of $X$. Then $\imath(\T)$ is also dually embedded in $X$ and by  Theorem \ref{mathsmath} the twisted sum splits.
\end{proof}
\begin{corollary}\label{theo_precompact}\label{theo_loc_compact}Locally compact Hausdorff Abelian groups and precompact Hausdorff Abelian groups belong to  $\ST$.


\end{corollary}

\begin{remark}
It is proved in \cite{domanski} that every topological vector space endowed with its weak topology is a $\mathcal{K}$-space. The above corollary shows that a similar result is true for topological Abelian groups, since a topological group endowed with the topology induced by its characters is precompact (see \cite{CR}).
\end{remark}

 \begin{theorem}\label{quotientst} Let $G$ be a topological Abelian group.
 \begin{enumerate}
 \item
If  a closed subgroup $H\leq G$  is such that $G/H \in \ST$, then $H$ is dually embedded.
\item
If  $G$ is in $\ST$ and $H\le G$ is a closed dually embedded subgroup, then $G/H\in \ST$.
 \end{enumerate}
\end{theorem}

\begin{proof}

(1) Suppose that $G/H\in\ST.$
Let $\chi :H \to \T$ be a character, and consider the natural twisted sum  $0\to H\stackrel{\imath}{\to} G \stackrel{\pi}{\to}G/H\to 0$. Taking the corresponding push out, we obtain the following commutative diagram
\[
   \xymatrix{
   0\ar[r]  & H\ar[d]^{\chi}\ar[r]^{\imath} & G\ar[r]^{\pi}\ar[d]^{s} & G/H\ar@{=}[d]   \ar[r]  & 0\\
   0\ar[r]  &  \mathbb{T}\ar[r]_{r}          &PO\ar[r]\ar@{-->}@/^{5mm}/[l]^{T}                                  &G/H \ar[r]    &0
   }
  \]
 where both rows are exact. Since $G/H\in \ST$,  the bottom sequence splits, hence by  Theorem \ref{theorem_section} there exists $T:PO \to \T$ such that $T\circ r=\text{id}_{\T}$. The homomorphism $T\circ s$ is an extension of $\chi$.

(2) Suppose that $H$ is dually embedded in $G$.
Fix a twisted sum $0\to \T\stackrel{\imath}{\to} X \stackrel{\pi}{\to}G/H\to 0$. Let $q:G \to G/H$ be the canonical projection and let $PB=\{(g,x)\in G\times X : q(g)=\pi (x)\}$ be the pull-back of $q$ and $\pi$. Define $r:PB \to G$ and $s: PB \to X$, as usually, as the restrictions of the corresponding projections. Note that
\begin{align*}
\text{ker} \,r &=\{(g,x)\in PB: r(g,x)=0\}=\{(g,x)\in PB: g=0\}\\
                    &=\{(0,x)\in G\times X : 0=\pi(x)\}\\
                     &=\{0\}\times \imath(\T),\\
\text{ker}\, s & =\{(g,x)\in PB: s(g,x)=0\}=\{(g,x)\in PB: x=0\}\\
                      &= \{(g,0)\in G\times X: q(g)=0\}\\
                      &=H\times \{0\}.
\end{align*}
We thus obtain the following diagram:

\[
\xymatrix{
0\ar[dr] &                                 &                                                &                               &0\ar[dl] \\
              & H\times \{0\}\ar@^{^(->}[dr]  &                                                 &\{0\}\times \imath(\T)\ar@_{(->}[dl]      &            \\
                    &                              &PB\ar[ld]_r\ar[dr]^s      &                                &            \\
                    &G\ar[dl]\ar[dr]_q &                                                 &X\ar[dl]^\pi\ar[dr]     &            \\
0                 &                               &G/H                                        &                                 &0
}
\]

Note that since $q$ and $\pi$ are onto, the definition of $PB$ yields $r((U\times X)\cap PB)\supset U$ and $s((G\times V)\cap PB)\supset V$ for every $U\in \mathcal N_0(G)$ and $V\in \mathcal N_0(X);$ hence $r$ and $s$ are onto and open.  Thus both short sequences are topologically exact.

 Fix $\varphi \in \imath(\T)^\wedge$; we will find an extension of $\varphi$ to the whole $X$. We can regard $\varphi $ as a character of  $\{0\}\times \imath(\T)$ by defining $\widetilde{ \varphi} (0,x)=\varphi(x)$. Since $\{0\}\times \imath(\T) \cong \T$ and $G\in \ST$, by Theorem \ref{mathsmath}, $\{0\}\times \imath(\T)$ is dually embedded in $PB$. Thus there exists $\psi\in PB^\wedge$ with $\psi|_{\{0\}\times \imath(\T)}=\widetilde{\varphi}$ i.e. $\psi (0,x)=\varphi(x)$ para todo  $x\in \imath(\T)$. Define, for every $h\in H$, $\widetilde{\psi}(h)=\psi (h,0)$ (note that if $h\in H$ then $(h,0)\in PB$).

Since $H$ is dually embedded in $G$ there exists $\sigma \in G^\wedge$ with $\sigma_{|H} =\widetilde{\psi}$ i.e. $\sigma (h)=\psi( h,0)$ for every $h\in H$. Now define $\rho\in PB^\wedge$ as follows: $\rho (g,x)=\psi (g,x)\overline{\sigma (g)}$. This is clearly continuous. Note that $\text{ker}\, \rho \geq \text{ker}\, s =H\times \{0\}$ since if $h\in H$, we have that $\rho (h,0)= \psi (h,0)\overline{\sigma (h)}=1$. As $X\cong PB / (H\times \{0\})$, the character $\widetilde{\rho}$ in $X^\wedge$ given by $\widetilde{\rho} (x)=\rho (g,x)$ for every $(g,x)\in PB$ is well defined and continuous.

Now $\widetilde{\rho}$ is the desired extension of $\varphi$: if $x\in \imath(\T)$, we have $\widetilde{\rho} (x)=\rho (0,x)=\psi(0,x)\overline{\sigma(0)}=\psi(0,x)=\varphi(x).$
\end{proof}
\begin{remark}
 An analogous result in the framework of $F$-spaces is  Theorem 5.2 in \cite{kaltonsampler} (cf. also Lemma 4.1 in \cite{domanski}).
\end{remark}

The following is a generalization of Theorem 5.1 in \cite{kaltonpeck} and appears as Theorem 4.1 in \cite{castillo}:

\begin{corollary}  A topological Abelian group $G$ belongs to  $\ST$ if and only if whenever $X$ is a topological Abelian group and $H$ is a closed subgroup of $X$ with $X/H \cong G$, then $H$ is dually embedded.
\end{corollary}
\begin{proof} This follows from Theorem \ref{quotientst}(1) and Theorem \ref{mathsmath}.
\end{proof}

Let $(G_\alpha)_{\alpha\in I}$ be a family of topological  Abelian groups. The coproduct of $(G_\alpha)_{\alpha\in I}$ is the direct sum $\bigoplus_{\alpha\in I} G_\alpha$
endowed with the finest group topology making the inclusion maps $i_\beta : G_\beta \hookrightarrow \bigoplus_{\alpha\in I}G_\alpha$ continuous, for every $\beta\in I.$
If  $(G_n)^\infty_{n=1}$ is a countable family of groups, this topology
 coincides with the box topology on $\bigoplus_{n\in \mathbb{N}} G_n$.

Recall that the coproduct $\bigoplus_{\alpha\in I} G_\alpha$ has the following universal property:
Given an arbitrary topological Abelian group $G$ and a homomorphism $f:\bigoplus_{\alpha\in I} G_\alpha  \to G$, $f$ is continuous if and only if $f\circ i_\beta$ is continuous $\forall \beta \in I$.

\begin{lemma}.
Let $G_\alpha\;(\alpha\in I)$ be topological Abelian groups and consider their coproduct $\bigoplus_{\alpha\in I}G_\alpha$. For every family of continous homomorphisms $f_\alpha:G_\alpha  \to X $ $(\alpha\in I)$, the map given by
\[
\begin{array}{cccc}
 f: &\displaystyle\bigoplus_{\alpha\in I} G_\alpha & \to &X\\
       &(g_\alpha)_{\alpha\in I}         &\longmapsto     &\sum_{\alpha\in I} f_\alpha (g_\alpha)
\end{array}
\]
is a continuous homomorphism.
\end{lemma}

\begin{proposition}\label{coproduct}
 Let $(G_\alpha)_{\alpha \in I}$ be a family of topological Abelian groups in $\ST$.
 The coproduct $\bigoplus_{\alpha\in I}G_\alpha$  belongs to $\ST.$
 \end{proposition}

 \begin{proof}
  Let
 $0\to \mathbb T\stackrel{\imath}{\to} X \stackrel{\pi}{\to}\bigoplus_{\alpha \in I} G_\alpha \to 0$
 be a twisted sum. Consider for each $\alpha \in I$, the pull-back $PB_{\alpha}$ of $\pi$ and $i_\alpha:G_{\alpha}\to \bigoplus_{\alpha \in I} G_\alpha$; for every $\alpha \in I$ there is a commutative diagram
  \[
   \xymatrix{
   0\ar[r]  & \mathbb{T}\ar@{=}[d]\ar[r] & PB_\alpha\ar[r]^{\pi_\alpha}\ar[d]^{v_\alpha} & G_\alpha \ar[d]^{i_\alpha}  \ar[r]  & 0\\
   0\ar[r]  &  \mathbb{T}\ar[r]          &X\ar[r]_{\pi}                                  &\displaystyle\bigoplus_{\alpha\in I}G_\alpha \ar[r]\ar[r]\ar@{-->}@/^{4mm}/[l]^{S}                                     &0
   }
  \]
with exact rows. As $G_\alpha\in \ST$, by Theorem \ref{theorem_section} there exist an homomorphism $S_\alpha:G_\alpha  \to PB_\alpha$ such that $\pi_\alpha\circ S_\alpha = \text{id}_{G_{\alpha}}$. Consider the map

\[
\begin{array}{cccc}
 S:   &\bigoplus_{\alpha\in I}G_\alpha  \to &X\\
      &(g_\alpha)_{\alpha\in I}        \longmapsto     &\displaystyle{\sum_{\alpha\in I}} v_\alpha \circ S_\alpha (g_\alpha)
\end{array}
\]
As $v_\alpha \circ S_\alpha $ is a continuous homomorphism, by the above Lemma, $S$ is a continuous homomorphism. For every $g=(g_\alpha)_{\alpha \in I} \in \bigoplus_{\alpha \in I} G_\alpha$ we have that
\[
\pi(S(g))  =\pi \Big(\displaystyle{\sum_{\alpha\in I}} v_\alpha \circ S_\alpha \left(g_\alpha\right)\Big)
                           =\displaystyle \sum_{\alpha\in I} (\pi \circ v_{\alpha} \circ S_{\alpha}) \left(g_{\alpha}\right)
                           =\displaystyle \sum_{\alpha\in I} i_{\alpha}\left(g_{\alpha}\right)=g
\]
so $S$ is a right inverse for $\pi$ and again by Theorem \ref{theorem_section}, the initial twisted sum splits.
\end{proof}

The class of nuclear groups was formally introduced by Banaszczyk in \cite{Ban}. His aim was to  find  a class of topological groups  embracing both nuclear spaces and locally compact Abelian groups (as natural generalizations of finite-dimensional vector spaces). The original definition is  rather technical, as could be expected from its strength of gathering objects from such different classes into the same framework. Next we collect some facts concerning the class of nuclear groups which are relevant to this paper:

\begin{itemize}
  \item Nuclear groups are locally quasi-convex $\cite[8.5]{Ban}$.
  \item Subgroups of nuclear groups are dually embedded $\cite[8.3]{Ban}$
  \item Products, countable coproducts,  subgroups and  Hausdorff quotients  of nuclear groups are nuclear $\cite[7.6, 7.8, 7.5]{Ban}$.
  \item Every locally compact Abelian group is nuclear $\cite[7.10]{Ban}$.
  \item A nuclear locally convex space is a nuclear group $\cite[7.4]{Ban}$. Furthermore, if a topological vector space is a nuclear group, then it is a locally convex nuclear space $\cite[8.9]{Ban}$.
\end{itemize}

 \begin{theorem}\label{limits}
Let $\{G_n, f^m_n\}$ be  a countable direct system  of nuclear Abelian groups in $\ST$.  Then the direct limit $\displaystyle \lim_{\rightarrow} G_n
$ belongs to $\ST$. In particular, sequential direct limits of locally compact groups belong to $\ST$.
\end{theorem}

\begin{proof}
The direct sum $\bigoplus^\infty_{n=1} G_n$ with the coproduct topology is in $\ST$ by \ref{coproduct}. Let $i_m  :G_m  \to \bigoplus G_n$ be the inclusion map, for every $m\in \mathbb N.$ It is known (see \cite{ATCh}) that $ {\displaystyle \lim_{\rightarrow}}\, G_n \cong (\bigoplus G_n )/ \overline{H}$, where  $\overline H$ is the closure of the subgroup $H$ generated by $\{ i_m \circ f^m_n (g_n) - i_n (g_n) : n \leq m; \; g_n \in G_n\}$ . Since countable coproducts of nuclear groups are nuclear groups,    $\overline{H}$ is dually embedded. By Theorem \ref{quotientst},   $\displaystyle \lim_{\rightarrow} G_n \in \ST$.
\end{proof}

Varopoulos introduced in \cite{Var} the class $\mathcal{L}_\infty$  of all topological groups whose topologies
are the intersection of a decreasing sequence of locally compact Hausdorff group
topologies.
He succeeded  in his aim  of   extending   known  results about  locally compact groups and established   the basis for the development of the harmonic analysis on $\mathcal{L}_\infty$
groups.  Subsequently many other authors
investigated different properties of this class (\cite{GaHer}, \cite{reade}, \cite{salvador}, \cite{sulley}, \cite{venkat}).

The following is a relevant fact concerning the structure of $\mathcal{L}_\infty$ groups proved by Sulley:

\begin{proposition}(\cite{sulley})
Let $G$ be any Abelian group endowed with an $\mathcal{L}_\infty$ topology.
Then it  has an
open subgroup which is a strict inductive limit of a sequence of Hausdorff locally
compact Abelian groups.
\end{proposition}

\begin{corollary}
Let $G$ be any Abelian group endowed with an $\mathcal{L}_\infty$ topology. Then $G$ belongs to $\ST$.
\end{corollary}	

\begin{proof}
By the above Proposition and Theorem \ref{limits},  $G$ has an open subgroup in $\ST$. Hence Corollary \ref{opensbg}
implies that  $G$ belongs to $\ST$.
\end{proof}

\section{Quasi-homomorphisms}\label{quasiquasi}

In his study of stability of homomorphisms between topological Abelian groups \cite{cabello}, Cabello defined the notion of quasi-homomorphism, which is inspired by the technique of quasi-linear maps introduced by Kalton and others (see \cite{kaltonsampler}).

\begin{definition}\label{aren}\cite{cabello} Let $G$ and $H$ be topological Abelian groups and $\omega:G\to H$ a map with $\omega(0)=0$. We say that $\omega$ is a {\em quasi-homomorphism} if the map $\Delta_{\omega}:(x,y)\in G\times G\mapsto \omega(x+y)-\omega(x)-\omega(y)\in H$ is continuous at $(0,0).$

 A quasi-homomorphism $\omega:G\to H$ is {\em approximable}  if there exists a homomorphism $a: G \to H$ such that $\omega-a$ is continuous at $0$.
\end{definition}

Our aim is to  use the notion of approximable quasi-homomorphisms to obtain new examples of groups in $\ST$.

We start with some facts about quasi-homomorphisms taken from \cite{cabello}:

\begin{proposition}\label{cabello}
Let $G$ and $H$ be topological Abelian groups and $\omega:G\to H$ be a quasi-homomorphism.
\begin{enumerate}
\item The sets \begin{multline*} W(V,U)=\{(h,g)\in H\times G\,:\,g\in U,\;h\in \omega(g)+V\}\\  (U\in{\cal N}_0(G),\;  V\in{\cal N}_0(H)) \end{multline*}
form a basis of neighborhoods of zero for a group topology on $H\times G$ .

 \item If $H\oplus_{\omega}G$ denotes the group $H\times G$ endowed with the topology induced by the quasi-homomorphism $\omega$ and $\imath_H$ and $\pi_G$ the canonical inclusion and projection, respectively,  $ 0\to H\stackrel{\imath_H}{\to} H\oplus_{\omega}G \stackrel{\pi_G}{\to}G\to 0$ is a twisted sum of topological Abelian groups.

 \item A quasi-homomorphism is approximable if and only if the induced twisted sum splits.

 \item The twisted sum $ 0\to H\stackrel{\imath}{\to} X \stackrel{\pi}{\to}G\to 0$ is equivalent to one induced by a quasi-homomorphism if and only if it splits algebraically and there exists a map $\rho: G \to X$ such that $\pi\circ \rho =\text{id}_G$,  $\rho (0)=0$ and $\rho$ is continuous at the origin.

 \end{enumerate}
 \end{proposition}

\begin{lemma}\label{solido}
Let $G$ and $H$ topological Abelian groups and $\omega:H \to G$, a quasi-homomorphism. Then the map $ g\in G\longmapsto (\omega(g),g) \in H\oplus_\omega G $ is continuous at $0$.
\end{lemma}
\begin{proof}
Note that for every $U\in{\cal N}_0(G)$ and $V\in{\cal N}_0(H)$ we have $g\in U\Rightarrow (\omega(g),g)\in W(V,U).$
\end{proof}
The following result is Lemma 11 in \cite{cabello}. We give here a proof for the sake of completeness.
\begin{proposition} \label{snork}
Let $\pi:X\to G$ be a continuous and open surjective homomorphism between Abelian groups. Suppose that $X$ is metrizable. Then there exists  $\rho:G \to X$ such that  $\pi\circ \rho=\text{id}_G$ and $\rho$ is continuous at $0\in G.$
\end{proposition}
\begin{proof} Note that to define $\rho$ with $\pi\circ \rho={\rm id}_G$, we simply must choose for every $g\in G$ an element $x\in \pi^{-1}(g),$ which is a nonempty set since $\pi$ is onto. Let us see that it can be done in such a way that the map thus obtained is continuous at zero.

Let $\{U_n : n\in \mathbb N\}$ be a decreasing basic sequence of neighborhoods of zero in $X$, where $U_1=X$. Due to the continuity of $\pi,$ we have $\bigcap_{n\in \mathbb{N}} \pi(U_n) = \{0\}$. 

Let $\rho$ take the value $0$ on $g=0$. For any $g\not=0$, by the previous paragraph we can  choose $n$ and $x$ with $\pi(x)=g,\;x\in U_{n},\;g\not \in \pi(U_{n+1}),$ and define $\rho(g)=x.$ Now fix $m\in \mathbb N;$ we must find $V\in {\cal N}_0(G)$ with $\rho(V) \subseteq U_m.$ Since $\pi$ is open there exists $V\in {\cal N}_0(G)$ with $\pi(U_m)\supseteq V.$  Fix $g\in V$ and let us show that $\rho(g)\in U_m.$ If $\rho(g)=0$ this is trivial. Otherwise $\rho(g)=x$ with $\pi(x)=g,\;x\in U_n,\;g\not \in \pi(U_{n+1})$ for some $n$. Then $g\in V \subseteq \pi(U_m),$ hence $m\le n$ and $x\in U_n \subseteq U_m.$

\end{proof}

\begin{corollary}\label{pi}

Let $G$ be a metrizable topological Abelian group.
\begin{enumerate}
\item
If $H$ is metrizable and divisible, every twisted sum  $ 0\to H\stackrel{\imath}{\to} X \stackrel{\pi}{\to}G\to 0$ is equivalent to one induced by a quasi-homomorphism.
\item $G$ belongs to $\ST$ if and only if every quasi-character $\omega: G \to \T$ is approximable.
\end{enumerate}
\end{corollary}

\begin{proof} Metrizability is a 3-space property (see \cite[5.38(e)]{hr}). If $0\to H\stackrel{\imath}{\to} X \stackrel{\pi}{\to}G\to 0$ is a topological extension and both $G$ and $H$ are metrizable, so   is $X$, and thus the hypotheses of Proposition \ref{snork} hold. Therefore, there exists an section of $\pi$ continuous at the origin. As $H$ is a divisible group, the twisted sum $ 0\to H\stackrel{\imath}{\to} X \stackrel{\pi}{\to}G\to 0$ splits algebraically. By Proposition \ref{cabello}(4) this twisted sum is equivalent to one induced by a quasi-homomorphism.

The second part is a consequence of the first one and Proposition \ref{cabello}.
 \end{proof}

\begin{definition} For any $t \in {\mathbb T}$ let us call $\langle t \rangle$ the only real number $\alpha \in \,]-1/2,1/2]$ such that $\exp(2\pi i \alpha)=t.$ Then $[t \rightarrow |\langle t \rangle|]$ is a norm on ${\mathbb T}$. (Note that $\langle ts \rangle\in  \langle t \rangle+\langle s \rangle+{\mathbb Z};$ hence  $|\langle ts\rangle|\le |\langle t \rangle+\langle s\rangle |\le |\langle t \rangle|+|\langle s\rangle |$.) Put $\theta(t)=|\langle t \rangle|.$ For every $\beta>0$ let us define $T_{\beta}=\{t\in \mathbb T \,:\, \theta(t)\le \beta\}.$ Note that $\T_+=T_{1/4}$ with this notation.
\end{definition}
\begin{definition} For every Abelian group $G$, every $V \subseteq G$ with $0\in V$ and every $n\in \mathbb N$ we define $(1/2^n)V=\{x\in V \,:\, 2^kx \in V \; \forall k \in \{0,1,\cdots,n\}\}.$
\end{definition}
\begin{proposition}\label{enes} (\cite{chdom}, Corollary 2) For every $n\in \mathbb N\cup\{0\}$ and $\beta\in [0,1/3),$ we have $(1/2^n)(T_{\beta})=T_{\beta/2^n}.$
\end{proposition}
\begin{lemma}\label{guararei}\label{metodo_secreto_xabi} Let $G$ be a topological Abelian group and let $\omega:G\to \mathbb T$ be a quasi-character. Suppose that there exist $U\in {\mathcal N}_0(G)$ and $\beta \in (0,1/3)$ such that $\omega (U) \subseteq T_{\beta}$. Then $\omega$ is continuous at zero.
\end{lemma}
\begin{proof} Since $\omega$ is a quasi-character, for every $\rho>0$ there exists $W_{\rho}\in \mathcal{N}_0(G)$ with $\omega(u)^2\overline{\omega(2u)}\in T_{\rho}$ for every $u\in W_{\rho}.$ Fix any $\varepsilon>0.$ Let us find $V\in \mathcal{N}_0(G)$ with $\omega (V) \subseteq T_{\varepsilon}.$

Fix any $\beta'\in(\beta,1/3).$ Find $N\in \mathbb N$ with $\beta'/2^N\le \varepsilon.$ Put $V=(1/2^N)U \cap (1/2^{N-1})W_{(\beta'-\beta)/(N2^{N-1})}.$

It is enough to prove that $$\forall v \in V \quad \forall n\in \{0,1,2,\cdots,N\} \quad \omega(v)^{2^n} \subseteq T_{\beta'}$$ since by  Proposition \ref{enes} this will imply $\omega(v)\in T_{\beta'/2^N} \subseteq T_{\varepsilon}.$ Now, for every $n\in \{0,1,2,\cdots,N\}$
\begin{eqnarray*}
\omega(v)^{2^n}&=& \omega(v)^{2^n}\overline{\omega(2v)}^{2^{n-1}}\omega(2v)^{2^{n-1}}= \big(\omega(v)^2\overline{\omega(2v)}\big)^{2^{n-1}}\omega(2v)^{2^{n-1}}\\
&=& \big(\omega(v)^2\overline{\omega(2v)}\big)^{2^{n-1}} \omega(2v)^{2^{n-1}}\overline{\omega(2\cdot 2v)}^{2^{n-2}}\omega(2\cdot 2v)^{2^{n-2}}\\
&=& \big(\omega(v)^2\overline{\omega(2v)}\big)^{2^{n-1}} \big(\omega(2v)^2 \overline{\omega(2\cdot 2v)}\big)^{2^{n-2}}\omega(2\cdot 2v)^{2^{n-2}}\\ &=&\cdots \\&=&\Big(\prod_{j=0}^{n-1} \big(\omega(2^jv)^2 \overline{\omega(2\cdot 2^jv)} \big)^{2^{n-j-1}}\Big) \omega(2^nv)
\end{eqnarray*}
Since $v\in (1/2^N)U,$ we have that $\omega(2^nv)\in T_{\beta}.$ Now, for every $j\in \{0,\cdots,n-1\}$ and every $n\in \{0,\cdots, N\}$, $\omega(2^jv)^2\overline{\omega(2\cdot 2^jv)}\in T_{(\beta'-\beta)/(N2^{N-1})},$  since $2^jv \in W_{(\beta'-\beta)/(N2^{N-1})}.$ Thus $\big(\omega(2^jv)^2 \overline{\omega(2\cdot 2^jv)} \big)^{2^{n-j-1}}\in T_{(\beta'-\beta)/N},$ and we deduce $\prod_{j=0}^{n-1} \big(\omega(2^jv)^2 \overline{\omega(2\cdot 2^jv)} \big)^{2^{n-j-1}} \in T_{\beta'-\beta}.$ This implies that $\omega(v)^{2^n}\in T_{\beta'}$.
\end{proof}
\begin{corollary} \label{unochi2}  A quasi-character $\omega:G\to \mathbb T$ is approximable if and only if there exist $U\in {\mathcal N}_0(G)$, $\beta \in (0,1/3)$ and an algebraic character $\chi\in \text{Hom}(G, \mathbb T)$ such that $(\omega \overline{\chi}) (U) \subseteq T_{\beta}.$ Moreover any such $\chi$ approximates $\omega.$
\end{corollary}

\begin{lemma}(\cite{cabello}, Lemma 6)\label{lema_important_cabello}
Let $G$ be an Abelian group (no topology is assumed), and let $\omega: G \to \T$ any mapping such that for some $\beta < 1/3$, $\omega (x+y) \overline{\omega (x)} \overline{\omega (y)} \in T_\beta$ $\forall x, y \in G$. Then there is a unique character $\chi:G \to \T$ such that $\omega (x) \overline{\chi(x)} \in T_\beta \;\forall x\in G$.
\end{lemma}

\begin{theorem}\label{the_l0}
Let $\mu$ be a nonatomic $\sigma$-finite measure on a set $\Delta$. Let be $L_0:=L_0 (\mu)$ the space of all measurable functions on $\Delta$ with the norm
\[
\| f\| := \int_\Delta \min \{1,\;|f(t)|\}\, d \mu (x)
\]
Every quasi-character $\omega : L_0  \to \T $ is approximable.
\end{theorem}
\begin{proof}
We can assume, without loss of generality, that $\mu$ is a probability (in other chase, choosing a probability $\nu$ with the same null sets as $\mu$, we have that $L_0(\mu)$ and $L_0(\nu)$ are linearly isomorphic, hence topologically isomorphic with respect the underlying group structures).

Let $\omega: L_0  \to \T$ be a quasi-character, fix $\beta < 1/3 $ and choose $\delta_0$ such that $\omega(f+g) \overline{\omega (f)} \overline{\omega (g)} \in T_\beta$ for every $f,\,g$ with $ \|f\|_0 , \|g\|_0 \leq \delta_0$.

Let $\Delta = \bigoplus ^r_{i=1} \Delta_i$ be a partition  into measurable sets, with $\mu (\Delta_i)\leq \delta_0$ for all $1\leq i \leq r$. Then $L_0=\prod^r_{i=1}L_0 (\Delta_i)$ as a topological direct product. For all $f\in L_0 (\Delta_i)$ we have that
\begin{align*}
\|f\|_0 &= \int_\Delta \min \{1,\;|f(t)|\}\, d \mu (x) \leq   \int_{\{t\in \Delta: f(t)\not = 0\}} 1d \mu (x)\\
            & = \mu \{t\in \Delta: f(t)\not = 0\}\leq \mu (\Delta_i) \leq \delta_0
\end{align*}
Call $\omega_i$ the restriction of $\omega$ to each $L_0 (\Delta_i)$. As $\omega_i (f+g) \overline{\omega_i (f)} \overline{\omega_i (g)} \in T_\beta$ for every $ 1\leq i\leq r$ and $f,g\in  L_0 (\Delta_i)$, we can apply lemma \ref{lema_important_cabello} to obtain unique characters $\chi_i : L_0 (\Delta_i) \to \T$ such that  $\omega_i (f) \overline{\chi_i(f)} \in T_\beta$ for every $f\in L_0 (\Delta_i)$. By \ref{metodo_secreto_xabi}, we have that each $\omega_i \overline{\chi_i}$ is continuous at the origin of $L_0 (\Delta_i)$ and thus, the character $\chi: L_0  \to \T$ given by $\chi(f)=\prod^r_{i=1} \chi_i (f_i)$ (where $f=\sum^r_{i=1} f_i, \; f_i\in L_0 (\Delta_i)$) approximates $\omega$ near the origin.
\end{proof}

\begin{corollary}\label{pitburn}
 $L_0$ belongs to $\ST$.
\end{corollary}
\begin{proof}
This follows from Theorem \ref{the_l0} and Corollary \ref{pi}, since $L_0$ is a metrizable group.
\end{proof}

\begin{example}
Let $L_0$ be as in \ref{the_l0}. Fix a discrete, nontrivial $D\leq L_0$ (e. g. a copy of $\Z$). Note that $L_0$ does not have any nontrivial continuous characters, and in particular $D$ is not dually embedded in $L_0$. Using Theorem \ref{quotientst} we deduce that $L_0/ D$ is not  in $\ST$. Since $L_0\in \ST,$ this example shows that being in $\ST$ is not preserved by local isomorphisms (compare with Corollary \ref{opensbg} above).
\end{example}

Let $(G, \tau)$ be a topological group. We say that the topology $\tau$ is a linear topology if it has a basis of neighborhoods of $G$ formed by open subgroups.

\begin{proposition} \label{linearat} Let $G$ be a topological Abelian group with a linear topology. Then every quasi-character of $G$ is approximable.
\end{proposition}
\begin{proof} Let $\omega:G\to \mathbb T$ be a quasi-character. There exists an open subgroup $U\le G$ such that $\omega(a+b)\overline{\omega(a)\omega(b)}\in \mathbb T_+$ for every $a,b\in U.$ Using Lemma \ref{lema_important_cabello}  we deduce that there exists an algebraic character $\chi:U\to \mathbb T$ with $\omega(u) \overline{\chi(u)}\in \mathbb T_+$ for every $u\in U.$ Now  Corollary \ref{unochi2} implies that any algebraic extension of $\chi$ approximates $\omega.$
\end{proof}

\begin{corollary} Let $G$ be a topological Abelian group with a metrizable linear topology. Then $G\in \ST$.
\end{corollary}
\begin{proof} This follows from Proposition \ref{linearat} and Corollary \ref{pi}.
\end{proof}

\begin{example}
Countable products of discrete abelian groups belong to $\ST$.
\end{example}


\begin{thebibliography}{99}

\bibitem{ATCh} S. Ardanza-Trevijano, M.J. Chasco. The Pontryagin duality of sequential limits of topological abelian groups. J. Pure Appl. Algebra 202 (2005),  11-21.

\bibitem{Ban} W. Banaszczyk.  { Additive Subgroups of Topological Vector Spaces},  Lecture Notes in Mathematics 1466. Springer-Verlag, Berlin Heidelberg, New York, 1991.

\bibitem{BruPei}M. Bruguera and E. Mart\'{\i}n-Peinador. Open subgroups, compact subgroups an Binz-Butzman reflexivity. Topology Appl. 72 (1996), 101-111.

\bibitem{brutka} M. Bruguera and M. G. Tkachenko. The three space problem in topological groups. Topol. Appl. 153  (2006), no. 13, pp. 2278-2302.

\bibitem{cabello2000}F. Cabello. Pseudo Characters and Almost Multiplicative Functionals. J. Math. Anal. Appl. 248 (2000), 275-289.

\bibitem{cabellojmaa} F. Cabello. Quasi-additive mappings. J. Math. Anal. Appl. 290 (2004), 263-270.

\bibitem{cabello} F. Cabello. Quasi-homomorphisms.  Fundam. Math. 178, 3 (2003), 255-270.


\bibitem{castillo} J. M. F. Castillo. On the ``three-space" problem for locally quasi-convex
topological groups. Arch. Math. 74 (2000), 253-262.


\bibitem{chdom} M.J. Chasco, X. Dom\'{\i}nguez. Topologies on the direct sum of topological abelian groups. Topol. Appl. 133 (2003) 209-223. 

\bibitem {CR} W. W. Comfort and  K. A. Ross.  Topologies induced by
groups of characters.  Fund. Math. 55 (1964) 283-291.



\bibitem{domanski} P. Doma\'nski. On the splitting of twisted sums, and the three space problem for local convexity. Studia Math. 82 (1985) 155-189.


\bibitem{GaHer} J. Galindo and S. Hern\'{a}ndez. The concept of boundedness and the Bohr compactification
of a MAP abelian group. Fund. Math. 159 (1999), 195-218.


\bibitem{hr}  E. Hewitt, K.A. Ross. Abstract harmonic analysis I. 2nd edition. Springer-Verlag, 1979.



\bibitem{kaltonpeck} N. J. Kalton and N. T. Peck. Quotients of $L_p(0,1)$ for $0\le p <1$. Studia Math. 64 (1979), 65-75.

\bibitem{kaltonsampler} N. J. Kalton, N. T. Peck, James W. Roberts. An F-space sampler. London Mathematical Society, Lecture note series 89 (1984).


\bibitem{morris}
S. A. Morris. Pontryagin Duality and the Structure of Locally Compact abelian Groups. London Mathematical Society Lecture Note Series 29.
\bibitem{noble} N. Noble. K-groups and duality. Trans. Amer. Math. Soc. 51 (1970), 551-561.



\bibitem{reade} J. B. Reade. A theorem on cardinal numbers asociated with inductive limits of locally compact
abelian groups. Math. Proc. Camb. Phil. Soc. 61 (1965), 69-74.

\bibitem{ribe} M. Ribe. Examples for the nonlocally convex three space problem. Proc. Amer. Math. Soc. 237 (1979) 351-355.
\bibitem{roberts} J. W. Roberts. A non locally convex $F$-space with the Hahn-Banach aproximation property, Banach spaces of analytic functions. Springer Lecture Notes 604, (1977) 76-81


\bibitem{salvador}  S. Hern\'{a}ndez.  A theorem on cardinal numbers associated with ${\mathcal L}_{\infty}$ abelian
groups. Proc. Camb. Phil. Soc 134 (2003) 33-39

\bibitem{sulley} L. J. Sulley. On countable inductive limits of locally compact abelian groups. J. London
Math. Soc. 2, 5 (1972), 629-637.

\bibitem{Var} N. Th. Varopoulos. { Studies in harmonic analysis}.  Proc. Camb. Phil. Soc. { 60} (1964) 465-516.
\bibitem{warner} S. Warner. Topological Fields. North-Holland, Mathematics Studies 157.


\bibitem{venkat} R. Venkataraman. Characterization, structure and analysis on abelian L1 groups. Mh.
Math. 100 (1985), 47-66.
\end{thebibliography}
\end{document}